\documentclass[a4paper,12pt]{article}
\usepackage{amsmath,amsthm,amssymb}
\usepackage[T2A]{fontenc}
\usepackage{amscd}

\usepackage{xcolor}

\usepackage{wrapfig}
\usepackage{graphicx}

  \newskip\prethm \prethm3.0pt plus1.3pt minus.4pt
  \newskip\posthm \posthm2.7pt plus1.4pt minus.3pt
  \newtheoremstyle{STATEMENT}%
       {\prethm}{\posthm}{\itshape}{\parindent}{\scshape}
       {.}{.6em plus.2em minus.1em}{}
  \newtheoremstyle{EXPLANATION}%
       {\prethm}{\posthm}{}{\parindent}{\scshape}
       {.}{.6em plus.2em minus.1em}{}
\theoremstyle{STATEMENT}

\newtheorem{theorem}{Theorem}
\newtheorem{proposition}{Proposition}

\newtheorem{corollary}{Corollary}

\theoremstyle{EXPLANATION}
\newtheorem{definition}{Definition}
\newtheorem{remark}{Remark}

\newcommand{\R}{\mathbb R}

\newcommand{\Z}{\mathbb Z}

\newcommand{\tr}{\mathrm{tr}\,}

\newcommand{\Tr}{\mathrm{tr}\,}

\newcommand{\so}{\mathrm{so}}

\newcommand{\gl}{\mathrm{gl}}

\newcommand{\uu}{\mathrm{u}}

\newcommand{\goth}{\mathfrak}

\newcommand{\SO}{\mathrm{SO}}

\newcommand{\GL}{\mathrm{GL}}

\newcommand{\OO}{\mathrm{SO}}

\newcommand{\Sym}{\mathrm{Sym}}

\medskip

\title{Argument shift method and sectional operators: applications to differential geometry}
\author{Alexey Bolsinov\\
School of Mathematics,
 Loughborough University,
 LE11 3TU, UK \\
Moscow State University, Moscow, Russia}

\begin{document}
\maketitle

\section*{What is this paper about?}

This text does not contain any new results, it is just an attempt to present, in a systematic way, one  construction which 
makes it possible to use some ideas and notions well-known in the theory of integrable systems on Lie algebras to a rather different area of mathematics related to the study of projectively equivalent  Riemannian and pseudo-Riemannian metrics.
The main observation can be formulated, yet without going into details, as follows:

\medskip

\noindent {\it The curvature tensors of projectively equivalent metrics coincide with the Hamiltonians of  multi-dimensional rigid bodies.}

\medskip

\noindent  Such a relationship seems to be quite interesting and may apparently have further applications in differential geometry.  The wish to talk about  this relation itself  (rather than some new results) was one of motivations for this paper.

The other motivation was to draw reader's attention to  the {\it argument shift method} developed by A.\,S.~Mischenko ana A.\,T.~Fomenko \cite{MF1}  as a generalisation of  S.\,V.~Manakov's construction \cite{Manakov}.  This method is, in my opinion,  a very simple, natural and universal construction which, due to its simplicity, naturality and universality,  occurs in different ares of modern mathematics. There are only a few  constructions in mathematics of this kind. In this paper, the argument shift method is only briefly mentioned but
the main subject, the so-called sectional operators, is directly related to it.

We discuss some new results obtained in our three papers
 \cite{BolsMatvKios, bolstsonev, BMR}.
In this sense, the present work can be considered as a review,  but I would like to shift the accent from the results to the way how using  algebraic properties of sectional operators helps in solving geometric problems.  That is why the exposition is essentially different from the above mentioned papers and details of the proofs,  which are not directly related to our main subjects, are omitted. Two first sections are devoted to the definition and properties of sectional operators,  in the following four, we discuss their applications  in geometry.   Also I would like to especially mention the note   \cite{BolsKonyaev},   in which we discussed the properties of sectional operators in a very general setting and which was conceptually very helpful for this paper.  I am very grateful to all of my coauthors,  Vladimir Matveev,  Volodymyr Kiosak,  Dragomir Tsonev,  Stefan Roseman and Andrey Konyaev. 
 
I wish to express special thanks to my teacher,  Anatoly Timofeevich Fomenko, without whom this work would never appear.

\section{Sectional operators on semisimple Lie algebras}\label{sect1}

We start with a brief  overview on (one special type of)  integrable Euler equations on semisimple Lie algebras (more details on this subject can be found in \cite{bogo, BolsBorisov, fomtrofbook,  fomsympl, marsden, MF1, pere,  reyman, troffomfact}).

Let  $\goth g$ be a semisimple Lie algebra, $R: \goth g \to \goth g$  an operator  symmetric with respect to the Killing form $\langle~,~\rangle$ on $\goth g$. The differential equation
\begin{equation}
\dot x = [R(x),x], \quad x\in\goth g,
\label{Euler}
\end{equation}
is Hamiltonian on $\goth g$ with respect to the standard Lie-Poisson structure  and is called the {\it Euler equation} related to the Hamiltonian function $H(x)=\frac{1}{2}\langle R(x), x\rangle$.

A classical, interesting  and extremely difficult problem is to find those operators
$R: \goth g \to \goth g$ for which the system (\ref{Euler}) is completely integrable.

One of such operators was discovered by S.\,Manakov in \cite{Manakov} and his idea then led to an elegant general construction developed by A.\,Mischenko and A.\,Fomenko \cite{MF1},  called the {\it argument shift method} and having many remarkable applications. In brief this construction for semisimple Lie algebras can be presented as follows.

Assume that $R : \goth g \to \goth g$ satisfies the following identity
\begin{equation}
[R(x),a]=[x,b],   \quad x\in \goth g,
\label{sectional}
\end{equation}
for some fixed $a,b\in \goth g$,  $a\ne 0$.   Then the following statement holds

\begin{theorem}[A.Mischenko, A.Fomenko \cite{MF1}]  \label{MFth}
Let $R: \goth g \to \goth g$ be symmetric and satisfy {\rm (\ref{sectional})}. Then
\begin{itemize}
\item the system {\rm (\ref{Euler})}
admits the following Lax representation with a parameter:
$$
 \frac{d}{dt}(x + \lambda a) = [ R(x)+\lambda b, x+\lambda a];
$$

\item the functions  $f(x+\lambda a)$, where $f:\goth g\to \R$ is an invariant of the adjoint representation,  are first integrals of {\rm (\ref{Euler})}  for any $\lambda\in \R$ and, moreover, these integrals commute;

\item if $a\in \goth g$ is  regular, then {\rm (\ref{Euler})} is completely integrable.
\end{itemize}
\end{theorem}

This construction has a very important particular case. If the Lie algebra $\goth g$ admits  a  $\Z_2$-grading,  i.e., a decomposition $\goth g = \goth h  + \goth v$  (direct sum of  subspaces) such that $[\goth h,\goth h]\subset \goth h$, $[\goth h,\goth v]\subset \goth v$, $[\goth v,\goth v]\subset \goth h$,
then  we may consider  $R:\goth h \to \goth h$  satisfying   (\ref{sectional}) with $a,b\in \goth v$,  and
Theorem   \ref{MFth} still holds if we replace $\goth g$ by $\goth h$.

The most important  example for applications (in particular,  in the theory of integrable tops)  is $\goth  g = \mathrm{sl}(n,\R)$,  $\goth h = \so(n,\R)$, with  $a$ and $b$ symmetric matrices. This  is  the  situation  that was studied in the pioneering work by S.~Manakov \cite{Manakov}  leading  to integrability of the Euler equations of $n$-dimensional rigid body dynamics.

From the algebraic point of view,  the above construction still makes sense if we replace $\so(n)$ by $\so(p,q)$ and  assume $a,b$ to be  symmetric operators with respect to the corresponding indefinite form $g$.
Moreover, if  we complexify our considerations we do not even notice any difference.
However, to indicate the presence (but not influence) of the bilinear form $g$, we shall denote the space of  $g$-symmetric operators by $\Sym(g)$, and the Lie algebra of $g$-skew-symmetric operators by $\so(g)$.

\begin{definition}
\label{defsect} 
{\rm  We shall say that $R: \so(g) \to \so(g)$ is a {\it sectional operator} associated with $A,B\in \Sym(g)$,  if  $R$ is symmetric with respect to the Killing form and the following identity holds:
\begin{equation}
[R(X), A] = [X,B],    \quad \mbox{for all } X\in  \so(g). 
\label{sectionall}
\end{equation}
}\end{definition}

We follow the terminology introduced by Fomenko and Trofimov  in \cite{fomtrofbook,troffom} where they studied various generalisations of such operators (see also \cite{MF1}, \cite{BolsKonyaev}). Strictly speaking, the above definition is just a particular case of  a more general construction.  The term ``sectional''  was motivated by the following reason.   The  identities
(\ref{sectional}) and (\ref{sectionall}) suggest  that one may represent $R$ as $\mbox{ad}_A^{-1}\mbox{ad}_B$,  but in general we cannot do so because $\mbox{ad}_A$, as a rule,  is not invertible. That is why  the operator $R$ splits into different parts each of which acts independently on its own subspace (section).  A similar partition of $R$ into ``sections'' will be seen in the proof of Proposition \ref{prop4}  below.          

\begin{remark}\label{twotypes}
In fact, there are two different types of sectional operators defined respectively by  \eqref{sectional} and  \eqref{sectionall}.  In this paper we focus on those defined by \eqref{sectionall} (Definition~\ref{defsect}). The first type, in some sense more natural and fundamental,  was introduced and  studied by Mischenko and Fomenko in  \cite{MF1, MischFom2}.  Traditionally, the operators from this class are denoted by $\varphi_{a,b,D}$,  they possess many interesting properties and applications too,  and we refer to
\cite{ Ratiu,  Bolsinov91, Konyaev, Konyaev2, troffom}   for examples and details.  
\end{remark}

In the next section we discuss basic properties of sectional operators in the sense of Definition \ref{defsect}.

\section{Algebraic properties of sectional operators}\label{sect2}

The first property is well-known. 

\begin{proposition} \label{prop1}
Let $R$ be a sectional operator associated with $A,B\in \Sym(g)$, i.e., satisfy \eqref{sectionall} for all $X\in \so(g)$. Then $A$ and $B$ commute. Moreover, $B$  lies in the center of the centralizer of $A$.  In particular, 
$B$ can be written as $B=p(A)$ for some polynomial $p(\cdot)$.
\end{proposition}

\proof
  Indeed,
$\langle [B,A], X \rangle = \langle A, [X,B] \rangle =\langle A, [R(X),A] \rangle =
\langle [A,A], R(X) \rangle =0 $ for any $X\in \so(g)$,  so $[A,B]=0$.   Moreover, if instead of $A$ we substitute any element  $\xi$  from its centralizer $\goth z_A=\{ Y~|~ [Y,A]=0\}$,  we obviously get the same conclusion
$[B,\xi]=0$, i.e., $B$  lies in the center of the centralizer of $A$.  Here, by $\langle~,~\rangle$ we understand the usual invariant form $\langle X, Y\rangle = \Tr XY$. 

The representation of $B$ as a polynomial in $A$ is a standard fact from matrix algebra:  the centre of the centraliser of every square matrix $A$ is generated by its powers $A^k$,  
 $k=0,1,2,\dots$.
\endproof

Given $A$ and $B=p(A)$, can $R$ be reconstructed from the relation \eqref{sectionall}? Let $R_1$ and $R_2$ be two operators satisfying  \eqref{sectionall}. Then we have 
$$
[R_1(X) - R_2(X), A]=0,
$$
which means that the image of $R_1-R_2$ belongs to the centraliser of $A$ in $\so(g)$, i.e. the subalgebra
$$
\goth g_A = \{ Y\in \so(g)~|~ [Y,A]=0\}.
$$
In other words, we see that  $R$ can be reconstructed from  \eqref{sectionall} up to an arbitrary operator with the image in $\goth g_A$.

Also we notice that $\goth g_A$ is an invariant subspace for $R$.  Indeed, if $X\in \goth g_A$, then   $X$ commute with $B=p(A)$,  therefore the right hand side of  \eqref{sectionall} vanishes and we have $[R(X), A]=0$, i.e. $R(X)\in \goth g_A$. 

From the algebraic viewpoint, these two properties mean that the induced operator
\begin{equation}
\label{inducedR}
\tilde R : \so(g)/\goth g_A \to \so(g)/\goth g_A 
\end{equation}
is well defined and can be uniquely reconstructed from \eqref{sectionall}.

\begin{remark}
\label{regularcase}
As an important particular case,  assume that $A$ is {\it regular} in the sense of the adjoint representation, i.e., the centraliser $\goth z_A$ of $A$ has minimal possible dimension.  It is well known that in this case the centraliser of $A$ is generated by its powers $A^k$.  Hence  $\goth g_A=\goth z_A\cap \so(g)$ is trivial, as all the elements of $\goth z_A$ are $g$-symmetric, whereas $\so(g)$ consists of $g$-skew-symmetric matrices.  Therefore  the sectional operator $R$ can be uniquely reconstructed from $A$ and $B$. Namely,  $R(X)={\rm ad}_A^{-1}{\rm ad}_B (x)$, a well-known formula in the theory of integrable systems on Lie algebras \cite{MF1}.  
\end{remark}

It is interesting to notice that in the general case (i.e., for arbitrary $A$),  there is another explicit formula for a partial ``solution'' of \eqref{sectionall}.

\begin{proposition} \label{prop2}
\label{aboutR0}
Let $B=p(A)$, then 
\begin{equation}
R_0(X) = \frac{d}{dt}\Big|_{t=0}\,p(A+tX)
\label{rx}
\end{equation}
is a sectional operator associated with $A$ and $B$.  In particular, if $A$ is regular, then  $R_0(X) = {\rm ad}_A^{-1}{\rm ad}_B (X)$ and it is a unique solution of \eqref{sectionall}.
\end{proposition}

\proof
Indeed,  $[p(A+tX), A+tX]=0$  and differentiating w.r.t. to $t$ gives
$$
0=\frac{d}{dt}\Big|_{t=0}\,[p(A+tX),A+tX]=\Bigl[\frac{d}{dt}\Big|_{t=0}\,p(A+tX),A\Bigr]+[p(A),X],
$$
i.e., $\bigl[\frac{d}{dt}\big|_{t=0}\,p(A+tX),A\bigr]=[X,B]$, as required. 

However, we also need to check that $R_0(X) \in \so(g)$, i.e.,   $R_0(X)^* = -R_0(X)$, where  $*$ denotes  ``$g$--adjoint'':
$$
g(L^*u,v)=g(u,Lv), \qquad  u,v\in V.
$$

Since  $A^{\ast}=A$, $X^{\ast}=-X$,
$(p(A+tX))^{\ast}=p(A^{\ast}+tX^{\ast})$ and $"\frac {d}{dt}"$ and $"\ast"$ commute, we have
$$
\begin{aligned}
R_0(X)^* & = \frac{d}{dt}\Big|_{t=0}\,p(A+tX)^*  =
\frac{d}{dt}\Big|_{t=0}\,p(A^*+tX^*)= \\
& =\frac{d}{dt}\Big|_{t=0}\,p(A-tX)=
-\frac{d}{dt}\Big|_{t=0}\,p(A+tX)= -R_0(X),
\end{aligned}
$$
as needed.  Thus, $R_0(X) \in \so(g)$.

Finally, we  check that $R_0$ is symmetric with respect to the Killing form.  Since the Killing form on $\so(g)$ is proportional to a simpler invariant form $\langle X,Y\rangle = \Tr XY$, we will use the latter for our verification.  Without loss of generality we may assume that $p(A)= A^k$  (the general case follows by linearity). Then
$$
R_0(X) = \frac{d}{dt}\Big|_{t=0}\,(A+tX)^k = A^{k-1} X + A^{k-2} X A + \dots + AXA^{k-2} + XA^{k-1},
$$
and we have
$$
\begin{aligned}
\langle R_0(X), Y \rangle &=  \Tr ( (A^{k-1} X + A^{k-2} X A + \dots + AXA^{k-2} + XA^{k-1})\cdot Y) =  \\
                          &=  \Tr ( X\cdot (YA^{k-1} + AYA^{k-2} + \dots + A^{k-2}YA + A^{k-1}Y)  ) = \langle X, R_0(Y) \rangle,
\end{aligned}
$$
as required. \endproof

Another interesting property of sectional operators is that they satisfy the Bianchi identity.  To see this,  we use the following natural identification of $\Lambda^2 V$ and $\so(g)$:
\begin{equation}
\label{identify}
 \Lambda^2 V \longleftrightarrow \so(g),      \quad    v\wedge u = v\otimes g(u) - u \otimes g(v).
\end{equation}
Here the bilinear form $g$ is understood as an isomorphism $g:V \to V^*$  between ``vectors'' and ``covectors''.
Taking into account this identification, we have the following

\begin{proposition}\label{prop3}
$R_0$ defined by \eqref{rx} satisfies the Bianchi identity, i.e.
$$
R_0(u\wedge v)w+R_0(v\wedge w)u+R_0(w\wedge u)v=0\quad \mbox{ for all } u,v,w \in V.
$$
\end{proposition}

\begin{proof}
It is easy to see that our operator  $R_0: \Lambda^2 V\simeq \so(g) \to \so(g)$  can be written as $R_0(X)=\sum_k C_kXD_k$, where $C_k$ and $D_k$ are some $g$-symmetric operators  (in our case these operators are some powers of $A$).  Thus, it is sufficient to check the Bianchi identity for operators of the form $X \mapsto CXD$.  

For $X=u\wedge v$ we have
$$
C(u\wedge v)Dw=  Cu \cdot g(v, Dw) - Cv \cdot g(u, Dw)
$$
Similarly, if we cyclically permute $u,v$ and $w$:
\begin{equation*}
C(v\wedge w)Du=  Cv \cdot g(w, Du) - Cw \cdot g(v, Du)
\end{equation*}
and
\begin{equation*}
C(w\wedge u)Dw=  Cw \cdot g(u, Dv) - Cu \cdot g(w, Dv).
\end{equation*}
Summing  these three expressions and  taking into account that $C$ and $D$ are $g$-symmetric, we obtain zero, as required. \end{proof}

One more useful property is related to the case when  $B=p(A)=0$,  for instance if $p(\cdot)=p_{\mathrm{min}}(\cdot)$ is the minimal polynomial for $A$.  This case seems to be meaningless, but the operator $R_0(X)$ defined by  \eqref{rx} turns out to be non-trivial  (as  the derivative of $p_{\mathrm{min}}(\cdot)$ does not vanish!).

\begin{proposition}
\label{prop4}
Let $p(A)=0$, then  the  image of $R_0$  defined by \eqref{rx} is contained in $\goth g_A$, the centraliser of $A$ in $\so(g)$:
$$
R_0(X)  \in  \goth g_A = \{ Y\in \so(g)~|~ [Y,A]=0\}.
$$
\end{proposition}

\proof
We know from Proposition \ref{aboutR0}, that  $R_0$ satisfies the relation $[R_0(X), A]=[X, p(A)]$.  Since $p(A)=0$, we get $R_0(X)\in\goth g_A$.
\endproof

\begin{remark}  Does the image of $R_0(X)=\frac{d}{dt}p_{\mathrm{min}}(A+tX)|_{t=0}$ coincides with $\goth g_A$?   The  answer depends on the structure of Jordan blocks related to each eigenvalue of $A$.   Recall first of all that for a regular matrix $A$,  the subalgebra $\goth g_A$ is trivial, so the question becomes interesting only for singular $A$'s.  A straightforward computation shows that for semisimple $A$ we have $\mathrm{Im}  R_0 = \goth g_A$.  This property still holds in a more general situation if in addition we assume that  each eigenvalue $\lambda$ of $A$ admits at most two Jordan blocks.  More precisely,  the ``bad'' situation is when $\lambda$ admits more than 2 blocks of a non-maximal size. For example, if $A$ has several $\lambda$-blocks of size $k$  and one $m<k$,  then we still have $\mathrm{Im}  R_0 = \goth g_A$.
\end{remark}

As shown above,  $R$ can be reconstructed from $A$ and $B$  modulo operators with images in $\goth g_A$. It is natural to ask a converse question.  Given a sectional operator $R: \so(g) \to \so(g)$,  can we reconstruct $A$ and $B$?

\begin{proposition}\label{prop5}
Assume that  $R: \so(g) \to \so(g)$ is symmetric and satisfies at the same time two identities:
\begin{equation}
[ R(X), A]=[X, B]\qquad \mbox{and}  \qquad [ R(X), A']=[X, B'],
\label{2sectiona}
\end{equation}
with  $A,B,A',B'\in \Sym(g)$.
If $A$ and $A'$  are not proportional (modulo the identity matrix),  then
$B$ is proportional to $A$  and, therefore,  $[R(X) - k\cdot X, A]=0$ for some $k\in \R$. Moreover, if $A$ is regular, then   $R  = k \cdot {\mathrm{id}}$.
\end{proposition}


\proof 
Notice that adding scalar  matrices to $A$ or $B$ does not change the equation, so we consider  $A\mapsto A+ c\cdot \mathrm{Id}$ and $B\mapsto B+ c\cdot \mathrm{Id}$ as trivial  transformations.
Without loss of generality we may then assume all these operators $A,A',B, B'$ to be trace free. Moreover,  we are allowed to complexify all the objects  so that instead of  $\so(g)$ and $\Sym(g)$  we may simply consider the spaces of symmetric and skew-symmetric complex matrices.

Let $y$ and $z$ be arbitrary symmetric matrices, then $[A', y], [A, z] \in \so(g)$
and we have:
$$
[ R([A', y]), A]=[[A',y], B], \quad  [ R([ A , z]), A']=[[A, z], B'].
$$
Since $R$ is symmetric with respect to the Killing form $\langle \ , \ \rangle$  we
have
$$
\langle [[A', y], B] , z\rangle= \langle [ R([A', y]), A], z\rangle =
\langle  R([A', y]), [ A , z] \rangle = \langle  [A', y], R([ A , z]) \rangle =
$$
$$
 \langle  y, [ R([ A , z]), A'] \rangle =
\langle  y , [[A, z], B'] \rangle = \langle [[B',y], A], z \rangle
$$
Since $z$ is an arbitrary symmetric matrix, we conclude that
\begin{equation}
[[A', y], B] =[[B',y], A].
\label{mainident}
\end{equation}

Similarly,  $ [[A, y], B'] =[[B,y], A']$.
Using the Jacobi identity, it is not hard to see that
$[B, A']=[A, B']$.

Rewriting (\ref{mainident}) as
$$
y (B'A - A' B) + (AB' - B A') y = B'yA+ AyB' - ByA' - A'yB
$$
and noticing that $[B, A']=[A, B'] $ implies $B'A - A'B
=AB' - B A'$,  we get
$$
y  T + T  y = B'yA+ AyB' - ByA' - A'yB
$$
where $T$ denotes $AB' - B A'$.

This formula can be understood as a relation between two  linear operators acting on the space of symmetric matrices   $y\in \Sym(g)$).
To get some consequences from this identity,  we take a ``kind of trace''. Recall that we consider $A', A, B', B, y, T$ as usual  symmetric (complex) matrices.

Instead of $y$ we substitute the symmetric matrix of the form  $e_i  v^\top + v e_i^\top$,  where $e_i$ and $v$ are vector-columns ($e_1,\dots, e_n$  is the standard (orthonormal) basis), then apply the result to  $e_i$ and take the sum over $i$. Here is the result:
$$
(e_i  v^\top + v e_i^\top) T e_i + T ( e_i  v^\top + v e_i^\top) e_i=
B' ( e_i  v^\top + v e_i^\top)A e_i + ...
$$
$$
e_i  (Tv, e_i) +  v (Te_i, e_i)+ Te_i (v,e_i) + Tv (e_i, e_i) =
B' e_i (A v, e_i) + B' v (A e_i, e_i) + ...
$$

Using obvious facts from Linear Algebra such as
$$
\sum_i(Te_i, e_i) = \tr \,T,  \quad  \sum_i(e_i,e_i)=n, \quad
\sum_i e_i (v, e_i)=v,
$$
we get
$$
Tv +  \tr \, T\cdot v  + Tv + n \cdot Tv  =
B' A v  +  \tr \, A\cdot B' v  + ...
$$
Taking into account that $A, A', B, B'$ are all trace free we have
$$
((n+2) T +   \tr\, T \cdot {\rm Id} ) v  =
(B' A   + AB' -   A' B - BA') v.
$$
Since $v$ is arbitrary and $T=B' A -   A' B = AB'  - BA'$, we finally get
$$
nT + \tr \, T \cdot {\rm Id }= 0,
$$
but this simply means that $T=0$. Hence we come to the identity of the form
\begin{equation}
B'yA+ AyB' = ByA' + A'yB.
\label{mainident2}
\end{equation}

It remains to use the following simple statement:
if  $A,  B,  A', B'$  are symmetric, $A\ne 0$ and (\ref{mainident2}) holds for any symmetric $y$,
then either    $B = k\cdot  A$, or  $A'= k\cdot A$ for some constant $k\in \R$. 

By our assumption, $A$ and $A'$ are not proportional, so we conclude that
$B= k\cdot A$ and therefore the identity $[ R(X), A]=[X, B]$ becomes $[ R(X)- k\cdot X, A]=0$, as needed.

Now assume that $A$ is regular.  Then  $\goth g_A=\{0\}$ (Remark \ref{regularcase}) and $[ R(X)- k\cdot X, A]=0$ implies 
$R(X)= k\cdot X$ for all $X\in\so(g)$, i.e.,
$R = k\cdot {\rm{id}}$, as was to be proved. \qed


\begin{remark}
A similar result  for sectional operators of the first type (see Remark~\ref{twotypes}) was proved by A.~Konyaev \cite{Konyaev}.
\end{remark}

The next statement describes the eigenvalues of sectional operators.  For regular and semisimple $A$, this fact is well known, see \cite{Manakov, MF1},  and our observation is  a natural generalisation of it.

\begin{proposition} 
\label{prop6}
Let $R: \so(g) \to \so(g)$ be a sectional operator associated with $A$ and $B=p(A)$,
where $p(\cdot)$ is a polynomial.  Let $\lambda_1, \dots, \lambda_s$ be the distinct eigenvalues of $A$. Then the numbers
$$
\frac{p(\lambda_i) - p(\lambda_j)}{\lambda_i - \lambda_j}, \qquad i\ne j,
$$
are eigenvalues of $R$.  Moreover, if $A$ possesses a non-trivial Jordan $\lambda_i$-block, then the number
$$
p'(\lambda_i)
$$
is an eigenvalue of $R$ too  (here $p'$  denotes the derivative of $p$).
\end{proposition}

\proof 
Since $R$ is not always uniquely defined,  we are not able to find all the eigenvalues of $R$ from $A$ and $B$.  However, we can find some of them,  namely those of the induced operator $\tilde R$, see \eqref{inducedR}.   Clearly,  the eigenvalues of $\tilde R$  form a part of the spectrum of $R$ and for our computations we may set $\tilde R = \tilde R_0$, where $R_0$ is explicitly defined by \eqref{rx}.

 Using \eqref{rx} we can easily describe a natural partition of $\so(g)$ into invariant subspaces of $R_0$ each of which, as we shall see later,  ``carries'' one single eigenvalue of $R_0$ only  (some of them may accidentally coincide, but generically our invariant subspaces are exactly generalised eigenspaces of $R_0$).

For simplicity we shall assume that all the eigenvalues of $A:V \to V$ are real.  The decomposition $V=\oplus_i V_{\lambda_i}$ into generalised eigenspaces of $A$ naturally induces the following decomposition of $\so(g)$ 
$$
\so(g) = \oplus_{i\le j} \goth m_{ij}
$$
where  $\goth m_{ij}$ (that can be understood as $V_{\lambda_i}\wedge V_{\lambda_j}$) is spanned by the operators of the form
$$
v\wedge u = v \otimes g(u) - u \otimes g(v) \in \so(g), \quad \mbox{with }  v\in V_{\lambda_i}, \ u\in V_{\lambda_j},
$$
where $g(u)\in V^*$ is the covector corresponding to $u\in V$ under the natural identification of $V$ and $V^*$ by means of $g$   (so that  $g(v,u)=\langle v, g(u)\rangle)$.

This decomposition becomes transparent in the matrix form if we use a basis adapted to the decomposition $V=\oplus_i V_{\lambda_i}$.  Then
$$
A =\begin{pmatrix}
A_1 & & & \\
& A_2 & & \\
& & \ddots & \\
& & & A_s
\end{pmatrix}, \quad 
g = \begin{pmatrix}
g_1 & & & \\
& g_2 & & \\
& & \ddots & \\
& & & g_s
\end{pmatrix}
$$
and $\so(g)$ can be written in a block form
$$
\so(g)=  \left\{ \begin{pmatrix}
M_{11} & M_{12} & \dots & M_{1s}\\
M'_{12} & M_{22} & \dots & M_{2s}\\
\vdots & \vdots & \ddots & \vdots \\
M'_{1s} & M'_{2s} & \dots & M_{ss}
\end{pmatrix} \right\},
$$
where $M_{ii} \in \so(g_i)$ (diagonal blocks),  the blocks $M_{ij}$, $i<j$ (above the diagonal)  are arbitrary and related to $M'_{ij}$ (below the diagonal)  as
$g_j M'_{ij} = -  M_{ij}^\top g_i$.  Then  $\goth m_{ii} = \so(g_i)\subset \so(g)$, i.e. consists of the diagonal block $M_{ii}$  while all the other blocks vanish and
$\goth m_{ij}$ consists of the pair of blocks $M_{ij}$ and $M'_{ij}$  ($i<j$), the others vanish.

The following facts can be easily verified and we omit details.

\begin{enumerate}
\item Each subspace $\goth m_{ij}$ is $R_0$-invariant.
\item The restriction of $R_0$ onto $\goth m_{ij}$ possesses a single eigenvalue, namely $\frac{p (\lambda_i) - p(\lambda_j)}{\lambda_i - \lambda_j}$.
\item The restriction of $R_0$ onto $\goth m_{ii}$ possesses a single eigenvalue, namely $p'(\lambda_i)$.
\end{enumerate}

The first  is straightforward. The next is based on the following simple fact from matrix algebra.  Let $B$ and $C$ be square matrices of sizes $k\times k$ and $m\times m$ respectively. Suppose that $\lambda$ and $\mu$ are eigenvalues of $B$ and $C$ respectively and $B$ and $C$ have no other eigenvalues.  Then the eigenvalue of the operator  $Y  \mapsto  C Y - Y B$ acting on $k\times m$-matrices $Y$, is unique and equal to $\lambda - \mu$.  The third statement require some easy computation.

Thus, we know all the eigenvalues of the operator $R_0:  \so(g) \to \so(g)$.
Recall that we are interested in the eigenvalues of the induced operator $\tilde R_0: \so(g)/\goth g_A \to \so(g)/\goth g_A$.  Under such a reduction, some of eigenvalues may disappear.   As the last step of the proof, we are going to explain that all of them survive.

To that end, we use another fact from linear algebra  (which explains which eigenvalues survive under reduction):

\medskip
 {\it Let  $\phi : V \to V$  be a linear operator with an invariant subspace $U\subset V$.   Let $\lambda$ be an eigenvalue of $\phi$ and $V_\lambda\subset V$ be the generalised $\lambda$-eigenspace of $\phi$.  Then $\lambda$ is an eigenvalue of the induced operator $\tilde\phi: V/U \to V/U$ if and only if $V_\lambda \not\subset U$.}
 \medskip

Thus, in order to show that the above eigenvalues of $R_0$ survive under reduction,  it is sufficient to check that $\goth m_{ij}$ and $\goth m_{ii}$ are not contained in $\goth g_A$.  To see this we need just to have a look at the structure of $\goth g_A$. It can be easily checked that $\goth g_A$  has the following  block-diagonal matrix form  (we use the same adapted basis as before):
\begin{equation}
\label{gA}
\goth g_A =\left\{     
X = \begin{pmatrix}  X_1 & & & \\ & X_2 & & \\ & & \ddots & \\ & & & X_s \end{pmatrix},    \quad  X_i \in \goth g_{A_i}   
\right\},
\end{equation}
where $\goth g_{A_i}$ is the centraliser of $A_i$ in $\so(g_i)$, i.e. $\goth g_{A_i} = \{ Y \in \so(g_i)~|~ YA_i = A_iY\}$.
More detailed description of $\goth g_A$ can be found  in \cite{bolstsonev}.

It is seen from this description that the subspace $\goth m_{ij}$ lies ``outside''  $\goth g_A$   and the intersection $\goth m_{ij} \cap \goth g_A$ is trivial so that  $\goth m_{ij}\not\subset \goth g_A$.

For $\goth m_{ii}$, the situation is different.  According to its definition, $\goth m_{ii}$  coincides with $\so(g_i)$ and therefore $\goth m_{ii}$ is contained in $\goth g_A$ (see \eqref{gA})  if and only if $\goth m_{ii} = \so(g_i) = \goth g_{A_i}$, i.e. the matrix $A_i$ commutes with all the $g_i$-skew symmetric matrices $Y\in \so(g_i)$.   This happens, however, if and only if $A_i$ is a scalar matrix, i.e. $A_i = \lambda_i \cdot \mathrm{Id}$.  Otherwise, $\goth g_{A_i}$ is strictly smaller than $\so(g_i)$.  
According to our assumptions  (see Proposition \ref{prop6}),  $A$ possesses a non-trivial Jordan $\lambda_i$-block, i.e. $A_i$ is not scalar.  Hence $\goth m_{ii}$ is not contained in $\goth g_A$ and therefore $\mu_i = p'(\lambda)$ is an eigenvalue of $\tilde R_0$ as needed. This completes the proof of Proposition \ref{prop6}.  \endproof

\begin{remark}
The above results can, more or less automatically, be  transferred to the case of operators $R: \uu(g,J) \to \uu(g,J)$ on the unitary Lie algebra  (in formula \eqref{sectionall} we take $X$ to be skew-hermitean and $A$ and $B$ hermitean).  This case corresponds to the natural $\Z_2$-grading  $\gl(n,\mathbb C) = \uu(g,J) \oplus  \mathrm{Herm}(g,J)$.
\end{remark}

Here is the summary of the above properties.   Let $R: \so(g) \to \so(g)$ be a sectional operator associated with $A,B\in \Sym(g)$. Then
\begin{itemize}
\item  $A$ and $B$ commute, moreover $B=p(A)$ for some polynomial $p(\cdot)$.
\item  $R_0(X)= \frac{d}{dt} |_{t=0}p(A + tX)$ is a sectional operator associated with $A$ and $B=p(A)$.  If $A$ is regular, then a sectional operator associated with $A$ and $B$ is unique and therefore $R=R_0$.
\item  $R_0$ satisfies the Bianchi identity.
\item  If $B=p(A)=0$,  e.g. if $p=p_{\mathrm{min}}$ is the minimal polynomial of $A$, then the  image of $R_0$  is contained in $\goth g_A = \{ Y\in \so(g)~|~ [Y,A]=0\}$, the centraliser of $A$ in $\so(g)$.  Moreover, if each eigenvalue of  $A$  possesses at most two Jordan blocks, then the image of $R_0$ coincides with $\goth g_A$.

\item 
Suppose that  $R$ is, at the same time, a sectional operator for another pair $A',B'\in\Sym(g)$.
If $A\ne \lambda A'+\mu \mathrm{Id}$,  then
$B$ is proportional to $A$  and, therefore,  $[R(X) - k\cdot X, A]=0$ for some $k\in \R$. Moreover, if $A$ is regular, then $R  = k \cdot {\mathrm{id}}$.

\item Let $\lambda_1, \dots, \lambda_s$ be  distinct eigenvalues of $A$ and $B=p(A)$. Then the numbers
$\dfrac{p(\lambda_i) - p(\lambda_j)}{\lambda_i - \lambda_j}$, $i\ne j$,
are eigenvalues of $R$.  Moreover, if $A$ possesses a non-trivial Jordan $\lambda_i$-block, then the number
$p'(\lambda_i)$ is an eigenvalue of $R$ too  (here $p'$  denotes the derivative of $p$).

\end{itemize}

\section{Projectively equivalent metrics:  curvature tensor as a sectional operator}\label{sect3}

\begin{definition}
{\rm Two metrics $g$ and $\bar g$ on the same manifold $M$ are called {\it projectively equivalent},  if they have the same geodesics considered as unparametrized curves.
}\end{definition}

In the Riemannian case the local classification of projectively equivalent pairs $g$ and $\bar g$ was obtained by Levi-Civita in 1896 \cite{Levi-Civita}.  For  pseudo-Riemannian metrics, this problem turned out to be much more difficult.  For the most important cases, local forms for $g$ and $\bar g$ were obtained in  \cite{Aminova2, Golikov, Petrov}, but the final solution has been obtained only recently \cite{BM, BMgluing,  Boubel, Boubel2}.

 In analytic form,  the projective equivalence condition for $g$ and $\bar g$  can be written in several equivalent ways. 
One of them is based on  the $(1,1)-$tensor $A=A(g,\bar g)$ defined by
  
\begin{equation}
\label{L}
A_j^i := { \left|\frac{\det(\bar g)}{\det(g)}\right|^{\frac{1}{n+1}}} \bar g^{ik}
 g_{kj},
\end{equation}
where ${\bar g}^{ik}$ is the contravariant inverse of
${\bar g}_{ik}$. Since the metric $\bar g$  can be  uniquely reconstructed from $g$ and $A$, namely:
 \begin{equation}
 \label{gandL}
 \bar g (\cdot\, , \cdot) = \tfrac{1}{{ |\det(A)|}}g(A^{-1}\cdot\, ,\cdot)  
 \end{equation}
 the condition that $\bar g$  is geodesically equivalent to $g$ can be written as a system of PDEs on the components of $A$.     From   the point of view of partial differential equations, $A$ is more convenient than $\bar g$ as the corresponding system of
   partial differential equations on $A$ turns out to be linear  \cite{Sinjukov}. In the index-free form, it  can be written as follows  (where $\ast$ means  $g-$adjoint):
\begin{equation}
\label{main2}
\nabla_u A =\frac{1}{2} \bigl(u\otimes d\tr A + (u\otimes d\tr A )^* \bigr).
\end{equation}

\begin{definition}{\rm
\label{compatibility}
We say that a (1,1)-tensor  $A$ is {\it compatible} with $g$, if  $A$ is $g$-symmetric, nondegenerate at every point and satisfies \eqref{main2} at any point $x\in M$ and for all tangent vectors $u\in T_xM$.  
}\end{definition}

 A surprising relationship between sectional operators and geodesically equivalent metrics is explained by the following observation. Notice, first of all, that due to its algebraic symmetries   (skew-symmetry with respect to $i,j$ and $k,l$  and symmetry with respect to permutation of pairs $(ij)$ and $(kl)$),   the Riemann curvature tensor $R_{ij,kl}$  can be naturally considered as
a symmetric operator  $R: \so(g) \to \so(g)$   (strictly speaking we need to raise indices $i$ and $k$  by means of $g$ to get the tensor of the form $R^{i \  k} _{\ j \  l}$).  Equivalently such an interpretation can be obtained by using identification \eqref{identify} of
$\Lambda^2 V$  and $\so(g)$.

Thus, in the (pseudo)-Riemannian case,  a curvature tensor can be understood as a linear map
$$
R: \so(g) \to \so(g).
$$
In this setting, by the way,   the symmetry $R_{ij,kl} = R_{kl,ij}$ of the curvature tensor amounts to the fact that $R$ is self-adjoint w.r.t. the Killing form,  and ``constant curvature''  means that $R= k \cdot \mathrm{Id}$, $k=\mathrm{const}$. So this point of view on curvature tensors is quite natural.

The following observation was made in \cite{BolsMatvKios}.

\begin{theorem}
\label{BMK}
If $g$ and $\bar g$ are projectively equivalent, then the curvature tensor of $g$ considered as a linear map
$$
R:  \so(g) \to \so(g)
$$
is a sectional operator, i.e., satisfies the identity
\begin{equation}
\label{SECT2}
[R(X), A] = [X, B] \quad  \mbox{for all } X\in \so(g)  
\end{equation}
with $A$ defined by \eqref{L} 
and $B$ being the Hessian of $\frac{1}{2}\tr A$, i.e. $B=\frac{1}{2}\nabla \mathrm{grad}\, \tr A$.
\end{theorem}

This result is, in fact,  an algebraic interpretation of  some equations on the components of curvature tensors of projectively equivalent metrics obtained in tensor form by A.\,Solodovnikov \cite{Sol}, see also   \cite{Sinjukov}.

\proof
Consider the compatibility condition for the PDE system  \eqref{main2}.   Namely, differentiate \eqref{main2}  by means of $\nabla_v$ and then compute $\nabla_v\nabla_u A - \nabla_u\nabla_v A - \nabla_{[v,u]} A$  in terms of $\tr A$:
$$
\nabla_v\nabla_u A - \nabla_u\nabla_v A - \nabla_{[v,u]} A =  [v\otimes g(u) - u\otimes g(v),  B]. 
$$
It remains to notice that the left hand side of this identity is $[R (u\wedge v), A]$.   Hence,  taking into account that bi-vectors $v\wedge u=v\otimes g(u) - u\otimes g(v)$ generate $\Lambda^2 V \simeq \so(g)$,    we get \eqref{SECT2} as required.
\endproof

Hence we immediately obtain a  strong obstruction to the existence  of a projectively equivalent partner.
\begin{corollary}
\label{cor1}
In order for $g$ to admit a projectively equivalent metric $\bar g$  (which is not proportional to $g$, i.e.,  $\bar g \ne \mathrm{const}\cdot g$), the curvature tensor of $g$ must be a sectional operator for some $A\ne  \mathrm{const}\cdot \mathrm{Id}$ and $B$. 
\end{corollary}
 
\begin{remark}
Some other links between projectively equivalent metrics and integrable systems are discussed in  \cite{benenti, MT, MT2}.
\end{remark}

\section{Projectively equivalent metrics:  Fubini theorem}\label{sect4}

Given a Riemannian metric $g$, how many geodesically equivalent metrics can $g$ admit?   Typically,  the answer is:  just metrics of the form $\bar g = \mathrm{const}\cdot g$ (this can be seen, for example, from Corollary \ref{cor1} which says that the algebraic structure of the curvature tensor of $g$ must be very special).  Levi-Civita classification theorem  gives a lot of non-trivial examples of projectively equivalent pairs $g$ and $\bar g$   (more precisely, two-parameter families of such metrics).   Can such a family be larger,  for example, three-parametric?  In the Riemannian case,  the following classical result of Fubini  \cite{Fubini1,Fubini2} clarifies the situation: if three essentially  different  metrics on an  $(n\ge 3)$-dimensional manifold $M$ share  the same unparametrized geodesics, and two of them (say, $g$ and $\bar g$) are strictly nonproportional (i.e., the roots of the characteristic polynomial $\det ( \bar g -\lambda g )$ are all distinct) at least at  one point, then they have constant sectional curvature.

Following \cite{BolsMatvKios}, we will say that two metrics $g$ and
$\bar g$ are \emph{strictly nonproportional} at a point $x\in M$, if
the $g$-symmetric (1,1)-tensor  $G=g^{-1}\bar g$ (or equivalently, the tensor $A$ defined by \eqref{L}), is regular in the sense of Remark \ref{regularcase}.  

If one of
the metrics is Riemannian, strict nonproportionality  means that all
eigenvalues of $G$ have multiplicity one and that was one of the key properties used by Fubini.   In the pseudo-Riemannian case,  this idea does not work as $G$ (and $A$) may have non-trivial Jordan blocks. However,  the conclusion of the Fubini theorem still holds for pseudo-Riemannian metrics.

\begin{theorem}[\cite{BolsMatvKios}] \label{thm1}  Let $g$, $\bar g$ and $\hat g$ be three geodesically  equivalent metrics on a connected manifold $M^n $ of dimension $n\ge 3$. Suppose   there exists a  point  at which
$g$ and $\bar g$ are strictly nonproportional,  and    a point  at which
  $g$, $\bar   g$ and $\hat g$ are linearly independent. Then, the metrics $g, \bar g $ and $\hat g$  have constant  sectional curvature.
\end{theorem}

\proof  We simply use the uniqueness property for sectional operators (see Proposition~\ref{prop5}). Assume that we have three geodesically equivalent metrics  $g$, $\bar g$, and $\hat g$ and choose a generic point $x\in M$.
Then by Theorem \ref{BMK}, the Riemann curvature tensor $R$ of the metric $g$ at $x\in M$ satisfies at the same time two identities:
\begin{equation}
[ R(X), A]=[X, B]\qquad \mbox{and}  \qquad [ R(X), A']=[X, B'].
\label{2sectional}
\end{equation}

Here we assume that $A$ and $A'$ are not proportional modulo the identity matrix (otherwise we would have  $\hat g= \lambda \bar g + \mu g$ which is not the case) and $A$ is regular due to strict non-proportionality of $g$ and $\bar g$.
From now on,  we may forget about the geometric meaning of $A,B, A', B'$ and start thinking of them as just certain $g$-symmetric operators.  After this we simply apply Proposition \ref{prop5} to conclude that $R = k(x)\cdot \mathrm{id}$, i.e.  the sectional curvature of $g$ is constant in all directions.   The fact that this constant $k$ does not depend of a point $x$, follows from the well-known fact  that  if  $\dim M\ge 3$ then $R = k(x)\cdot \mathrm{id}$ implies that $k(x)=\mathrm{const}$.  

This gives a proof in local setting, i.e. in a neighbourhood of a generic point where the above mentioned  algebraic conditions on $A$ and $A'$ are satisfied.  
The fact that the set of such points is open and everywhere dense in $M$ is not obvious and needs additional arguments (see \cite{BolsMatvKios}).  \endproof

\section{New holonomy groups in pseudo-Riemannian geometry}\label{sect5}

In this section, we discuss the results and ideas developed in \cite{bolstsonev}.

Let $M$ be a smooth manifold endowed with an affine symmetric connection $\nabla$.  Recall that the {\it holonomy group of $\nabla$} is a subgroup
$\mathrm{Hol}(\nabla) \subset \GL(T_x M)$ that consists of the linear operators $A:T_x M \to T_x M$ being  ``parallel transport transformations''
along closed loops $\gamma$ with $\gamma(0)=\gamma(1)=x$.

Holonomy groups were introduced by \'Elie Cartan in the twenties  \cite{Ca1, Ca2} for the study of Riemannian symmetric spaces  and since then the classification of holonomy groups has remained one of the classical problems in differential geometry. The fundamental results in this direction are due to Marcel Berger \cite{Ber} who initiated the programme of classification of Riemannian  and irreducible holonomy groups which was  completed by  
D.~V.~Alekseevskii \cite{Alexeevski}, 
R.~Bryant \cite{Br1, Br2},  D.~Joyce \cite{Joy1, Joy2, Joy3},  L.~Schwachh\"ofer, S.~Merkulov \cite{MerkSchw}.  Very good historical surveys can be found in \cite{Br4, Sch2}.

The classification of Lorentzian holonomy groups has recently been obtained by T.~Leistner \cite{Le1} and A.~Galaev \cite{Gal2}. However, in the general pseudo-Riemanian case,  the complete  description of holonomy groups is a very difficult problem which still remains open,  and even particular examples  are of interest (see \cite{Bergery,  Boubel3, Gal1, Gal3, Ike}). We refer to \cite{GalaevLeistner}  for more information on recent development in this field.

The following theorem describes a new series of holonomy groups on pseudo-Riemannian manifolds.  As we shall see, the proof of this result essentially uses the concept and properties of sectional operators.

\begin{theorem}[\cite{bolstsonev}]
\label{main}
For every $g$-symmetric operator $A:V\to V$, the identity connected component $G^0_A$ of  its centraliser   in $\OO(g)$
$$
G_A=\{ X\in \OO(g)~|~ XA=AX\}
$$
is a holonomy group for a certain (pseudo)-Riemannian metric $g$.
\end{theorem}

Notice that in the Riemannian case this theorem becomes trivial:  $A$ is diagonalisable and the connected component  $G^0_A$ of its centraliser  is isomorphic to the standard direct product $\SO(k_1)\oplus \dots \oplus \SO(k_m) \subset  \SO(n)$, $\sum k_i\le n$,  which is, of course,  a holonomy group.  In the pseudo-Riemannian case,  $A$ may have non-trivial Jordan blocks  (moreover, any combination of Jordan blocks is allowed)  and the structure of $G^0_A$ becomes more complicated.


\proof

We follow the traditional approach to the problem of description of holonomy groups based on the notion of a Berger (sub)algebra.

\begin{definition} {\rm A map $R: \Lambda^2V \to \gl(V)$ is called a {\it formal curvature tensor} if it satisfies the Bianchi identity
\begin{equation}
R(u\wedge v)w + R(v\wedge w)u + R(w \wedge u)v = 0 \qquad \mbox{for all} \ u,v,w\in V.
\end{equation}
}
\end{definition}

This definition simply means that $R$ as a tensor of type $(1,3)$  satisfies all usual algebraic properties of  curvature tensors:
$$
R_{k\, ij}^m=R_{k\, ji}^m \quad \mbox{and}  \quad R_{k\, ij}^m+R_{i\, jk}^m+R_{j \, ki}^m=0.
$$

\begin{definition}\label{defBerger}  Let $\goth h\subset \gl(V)$ be a Lie subalgebra.  Consider the set of all formal curvature tensors $R: \Lambda^2V \to \gl(V)$ such that $\mathrm{Im}\, R\subset \goth h$:
$$
\mathcal R(\goth h)=\{ R: \Lambda^2V \to \goth h ~|~
R(u\wedge v)w + R(v\wedge w)u + R(w \wedge u)v = 0, \   u,v,w\in V\}.
$$
We say that $\goth h$ is a {\it Berger algebra} if it is generated as a vector space by the images of the formal curvature tensors $R\in \mathcal R(\goth h)$, i.e.,
$$
\goth h = \mathrm{span} \{ R(u\wedge v)~|~ R\in \mathcal R(\goth h), \ u,v\in V\}.
$$
\end{definition}

Berger's test  (which is sometimes referred to as Berger's criterion) is the following property of holonomy groups:

\medskip
{\it Let $\nabla$ be a  symmetric affine connection on $TM$.  Then the Lie algebra $\goth{hol} \, (\nabla)$ of its holonomy group $\mathrm{Hol} \, (\nabla)$ is {\it Berger}.}
\medskip

Usually the solution of the description problem for holonomy groups consists of two parts. First,  one tries to describe all Lie subalgebras $\goth h\subset \gl(n,\R)$ of a certain type satisfying Berger's test  (i.e., Berger algebras). This part is purely algebraic. The second (geometric) part is to find
a suitable connection $\nabla$ for a given Berger algebra $\goth h$ which realises  $\goth h$ as the holonomy Lie algebra, i.e., $\goth h=\goth{hol}\, (\nabla)$.

 We follow the same scheme but will use, in addition,  some ideas from projective differential geometry. As a particular case of projectively equivalent metrics $g$ and $\bar g$ one can distinguish the following.

\begin{definition}
{\rm  Two metrics $g$ and $\bar g$ are said to be {\it affinely equivalent}  if
their geodesics coincide  as parametrized curves.
}\end{definition}

It is not hard to see that this condition simply means that the Levi-Chivita connections $\nabla$ and $\bar\nabla$ related to $g$ and $\bar g$ are the same, i.e., $\nabla=\bar\nabla$ or,  equivalently,
$$
\nabla \bar g = 0.
$$

If instead of $\bar g$ we introduce a linear operator $A$  (i.e.  tensor field of type $(1,1)$) using the standard one-to-one correspondence $\bar g \leftrightarrow A$ between symmetric bilinear forms and $g$-symmetric operators:
$$
\bar g(\xi, \eta) = g(A\xi, \eta),
$$
then the classification of affinely equivalent pairs $g$ and $\bar g$  is equivalent to the classification of pairs  $g$ and $A$, where $A$ is covariantly constant w.r.t. the Levi-Civita connection $\nabla$ related to $g$ \footnote{The classification of such pairs has been recently obtained by C.~Boubel \cite{Boubel}.}.

On the other hand, the existence of a covariantly constant  $(1,1)$-tensor field $A$  can be interpreted in terms of the holonomy group  $\mathrm{Hol}(\nabla)$:
\medskip

{\it The connection $\nabla$ admits a covariantly constant $(1,1)$-tensor field if and only if
$\mathrm{Hol}(\nabla)$  is  a subgroup of the centralizer of $A$ in $\SO(g)$}:
$$
\mathrm{Hol}(\nabla)\subset G_A = \{ X\in \SO(g)~|~  XAX^{-1}=A\}.
$$

In this formula, by $A$ we understand the value of the desired $(1,1)$-tensor filed at any fixed point $x_0\in M$.  Since $A$ is supposed to be covariantly constant,  the choice of $x_0\in M$ does not play any role.

It is natural to conjecture that for a generic metric $g$ satisfying $\nabla A=0$, its holonomy group coincides with $G_A$ (or its identity component) exactly.  That is just another interpretation of the statement of our theorem. In other words, we want to construct (local) examples of  pseudo-Riemannian metrics  that admit covariantly constant (1,1)-tensor fields with a given algebraic structure, and to check that their holonomy group is the largest possible,  i.e. coincides with $G_A^0$.  As usual, it will be more convenient to deal with the corresponding Lie algebra $\goth g_A$.

If we formally apply  Theorem~\ref{BMK} to affinely equivalent metrics $g$  and $\bar g$  (or equivalently to the pair $g$, $A$)\footnote{The operator $A$ we use in this section is slightly different from that in Section \ref{sect3}, but the final conclusion will be the same.} and use the fact that $\tr A = \mathrm{const}$, we will see that   $R$  satisfies a simpler equation
$$
[R(X), A] =0,
$$
which, of course, directly follows from $\nabla A=0$ and seems to make all the discussion above not relevant to our very particular situation. 
However,  as we know from Proposition~\ref{prop4}, formula \eqref{rx}   still defines a non-trivial operator, if  $p(t)$ is a non-trivial  polynomial satisfying $p(A)=0$, for example,  the minimal polynomial $p_{\mathrm{min}}(\cdot )$ for  $A$.

Thus, this discussion gives us a very good candidate for the role of a formal curvature tensor  to verify the condition of Berger's test. Indeed, consider the sectional operator (associated with the given $A$ and $B=0$)  defined by 
\begin{equation}
R:\so(g) \to \so(g),  \qquad R(X) = \frac{d}{dt}\left. p_{\mathrm{min}}(A+tX) \right|_{t=0} 
\label{rxmin1}
\end{equation}

Using the natural identification  \eqref{identify}  of $\Lambda ^2V$ with $\so(g)$ and
Proposition \ref{prop3}, we see immediately that this operator is  a formal curvature tensor.  According to Proposition~\ref{prop4},  the image of this operator belongs to $\goth g_A$ and, moreover, coincides with $\goth g_A$ if $A$ satisfies certain algebraic conditions, in particular, if for each of eigenvalue of $A$ there are at most two Jordan blocks.  In the context of Berger's criterion,  this means that under these additional assumptions on $A$, the algebra $\goth g_A$ is Berger.  

To prove this result for an arbitrary $A$, it is sufficient to use the $g$-orthogonal decomposition $V=\oplus V_\alpha$ of $V$ into invariant subspaces corresponding to the Jordan blocks $J_\alpha$'s of $A$.  Such a decomposition  always exists, see \cite{lancaster, leep}, and it induces a natural partition of $\so(g)=\oplus_{\alpha \le \beta} \goth v_{\alpha\beta}$ into invariant subspaces of $R$  (similar to the partition $\so(g) = \oplus_{i\le j} \goth m_{ij}$ from Proposition \ref{prop6} and, more precisely, a subpartition of it).  After this one can continue working with each pair of Jordan blocks separately and construct an operator 
$$
R_{\alpha\beta} : \so (g, V_{\alpha}\oplus V_\beta) \to \so (g, V_{\alpha}\oplus V_\beta)
$$ 
by using the same formula \eqref{rxmin1} with the minimal polynomial of the matrix $A|_{V_{\alpha}\oplus V_\beta}$ consisting just of these two Jordan blocks.  
This operator, $R_{\alpha\beta}$ then can be naturally extended to the whole algebra $\so(g)$ by letting it to be zero on the natural complement of  $\so (g, V_{\alpha}\oplus V_\beta)$ in $\so(g)=\so(g, V)$.

Finally we set:
\begin{equation}
R_{\mathrm{formal}} = \sum_{\alpha\le \beta} R_{\alpha\beta}:\so(g) \to \so(g),  
\label{rxmin2}
\end{equation}
The operator so obtained is just a block-wise modification of \eqref{rxmin1},  the only difference is that now the minimal polynomial is appropriately chosen for each particular invariant subspace $\goth v_{ij}$. 

\begin{proposition}
\label{prop:berger}
The operator $R_{\mathrm{formal}}$ defined by \eqref{rxmin2} is a formal curvature tensor whose image coincides with $\goth g_A$.
In particular, the Lie algebra $\goth g_A$ is Berger.
\end{proposition}

The next step is a geometric realisation of this Berger algebra.
In other words,  for a given operator $A: V \to V$,  where $V$ is identified with the tangent space of a manifold $M$ at some fixed point $x_0$, we need to find a (pseudo)-Riemannian metric $g$ on $M$ and a $(1,1)$-tensor field $A(x)$  (with the initial condition $A(x_0)=A$) such that
\begin{enumerate}
\item $\nabla A(x)= 0$;
\item $\mathfrak{hol}\,(\nabla)=\goth g_A$.
\end{enumerate}

Notice that the first condition guarantees that $\mathfrak{hol}\,(\nabla) \subset \goth g_A$. On the other hand,  it is well known  (Ambrose-Singer theorem) that the image of the curvature operator $R_g(x_0)$ is contained in 
$\mathfrak{hol}\,(\nabla)$. Thus, taking into account Proposition \ref{prop:berger}, the second condition can be replaced by
\medskip

$2'$. $R_g(x_0)$ coincides with the formal curvature tensor   $R_{\mathrm{formal}}$ \eqref{rxmin2}.

\medskip

Thus, our goal  is to construct  (at least one example of)  $A(x)$ and $g(x)$ satisfying conditions 1 and $2'$.  To that end, we are going to use some special ansatz for $A$ and $g$. Namely we will assume that $A(x)$ does not depend on $x=(x_1,\dots, x_k)$ at all  (as was proved by A.\,P.~Shirokov  \cite{Shirokov}, such a coordinate system always exists if $\nabla A=0$), i.e.,
$$
A(x)=A=\mathrm{const}
$$
and $g$ is quadratic in $x$,  more precisely,
\begin{equation}
\label{eq0}
g_{ij} (x)= g_{ij}^0 + \sum  \mathcal B_{ij, pq}x^p x^q
\end{equation}
where $ \mathcal B$ satisfies obvious symmetry relations, namely,  $ \mathcal B_{ij,pq}= \mathcal B_{ji,pq}$ and $ \mathcal B_{ij,pq}= \mathcal B_{ij, qp}$.

Thus, our goal is to find $\mathcal B_{ij,pq}$.  It will be more convenient for us to replace $\mathcal B_{ij,pq}$ with $B^{i\ , p}_{\ j, \ q}=
g_0^{i\alpha} g_0^{p\beta}\mathcal B_{\alpha j,\beta q}$ and consider  this $B$ as a linear map
$$
B: \gl(V) \to \gl(V) \  \mbox{defined by }   B(X)^i_q = B^{i\ , p}_{\ j, \ q} X_p^j,
$$
where $V$ is understood as the tangent space at the origin $x_0=0$.

We want $g$ defined by \eqref{eq0} to satisfy the following three conditions:
\begin{enumerate}
\item $A$ is $g$-symmetric;
\item $\nabla A=0$;
\item $R_g(x_0)=R_{\mathrm{formal}}$,  where $x_0=0$ in our local coordinates.
\end{enumerate}
It can be easily checked that in terms of $B$, these conditions can be rewritten  respectively  as
\begin{equation}
\label{gsymm}
A B(X) = B(AX)\quad\mbox{for any}\quad  X\in \gl(V),
\end{equation}
\begin{equation}
\label{nablaA}
[B(X), A] + [B(X), A]^* = 0\quad\mbox{for any}\quad  X\in \gl(V),
\end{equation}
\begin{equation}
\label{realcurv}
R_{\mathrm{formal}} (X) = -B(X) + B(X)^*,   \quad X\in \so(g,V).
\end{equation}

The last formula \eqref{realcurv}, in fact,  shows  that $B$ can be understood as the extension of $R_{\mathrm{formal}}$ from $\so(g,V)$ to $\gl(V)$
(with factor $-\frac{1}{2}$).   In our case, such a natural extension indeed exists and can be defined by the formal expression  $B=-\frac{1}{2}R_{\mathrm{formal}} (\otimes)
$ which can be explained as follows.  Assume for simplicity that $R_{\mathrm{formal}}$ is defined by \eqref{rxmin1} with
$p_{\mathrm{min}}(t) = \sum_{m=0}^n a_m t^m$.  Then $R_{\mathrm{formal}}(X)$ can be written as
$$
 \frac{d}{dt} \big| _{t=0}\left( \sum_{m=0}^n a_m (A+t\cdot X)^m \right)= 
\sum_{m=0}^n a_m \sum_{j=0}^{m-1} A^{m-1-j} X A^{j},
$$
If in this expression we formally  substitute $\otimes$ instead of $X$ (and use the factor of $-\frac{1}{2}$), we obtain a desired tensor of type $(2,2)$:
\begin{equation}
B = - \frac{1}{2}\cdot \sum_{m=0}^n a_m \sum_{j=0}^{m-1} A^{m-1-j} \otimes  A^{j}.
\label{r3}
\end{equation}

Notice that $B(X)$ for $X\in \gl(V)$ is obtained from this this expression by replacing back $\otimes$ with $X$.
After this remark, the verification of \eqref{gsymm}, \eqref{nablaA}, \eqref{realcurv} is straightforward\footnote{It is interesting to notice that  \eqref{nablaA}
follows immediately from Proposition \ref{prop4} as in its proof we did not use that fact that $X$ was skew-symmetric, the conclusion of Proposition \ref{prop4} still holds for any $X\in \gl(V)$.} and  the realisation part is completed.   However,  in general,  $R_{\mathrm{formal}}$ is a combination of operators $R_{\alpha\beta}$ related to each pair of Jordan blocks of $A$.  But this does not represent any serious difficulty because we can use the same idea and  set $B=\sum_{\alpha\le\beta} B_{\alpha\beta}$  where $B_{\alpha\beta}$  are the tensors constructed from $R_{\alpha\beta}$. Since the equations \eqref{gsymm}, \eqref{nablaA}, \eqref{realcurv} are linear in $B$ in the natural sense,  the conclusion, we need, will obviously hold for the sum $B=\sum_{\alpha\le\beta} B_{\alpha\beta}$. Geometrically,  $B_{\alpha\beta}$ defines a direct product metric $g_{\alpha\beta}\times g_{\mathrm{flat}}$, where  $g_{\alpha\beta}$ is the metric on the sum of the subspace $V_\alpha\oplus V_\beta$ corresponding to the chosen pair of Jordan blocks and $g_{\mathrm{flat}}$ is the flat metric on the orthogonal complement to $V_\alpha\oplus V_\beta$ whose components are  constant in our local coordinates. This completes the proof.  \qed

\section{On the Yano-Obata conjecture for c-projective vector fields}\label{sect6}

In the paper \cite{BMR} we use sectional operators for studying global properties of c-projectively equivalent metrics. 
I would like to briefly mention some of our observations here as they could possibly lead to further applications of sectional operators in geometry.

\begin{definition}
{\rm
A curve  $\gamma (t)$ on a K\"ahler manifold $(M, g, J)$ is called  $J$-planar, if
$$
\nabla_{\gamma} \dot \gamma = \lambda \dot \gamma,
$$
where $\lambda \in \mathbb C$ is a complex number (depending on $t$), or equivalently
$$
\nabla_{\gamma} \dot \gamma = \alpha \dot \gamma + \beta J\dot\gamma
$$
where $\alpha,\beta\in\R$ , and $J$ is the complex structure on $M$.
}
\end{definition}

\begin{definition}
{\rm
Two K\"ahler metrics $g$ and $\hat g$ on a complex manifold $(M,J)$ are called {\it c-projectively equivalent}, if they have the same $J$-planar curves.
}
\end{definition}

The properties of c-projectively equivalent metrics are in many ways similar to those of metrics that are projectively equivalent in usual sense (cf. Section~\ref{sect3}).
By analogy with \eqref{L}, we can introduce a linear operator 
$$
A  = \left(\frac{\det \hat g}{\det g}\right)^{\frac{1}{2(n+1)}} \cdot \hat g^{-1} g,
$$
where $n=\dim_{\mathbb C}M$. Equivalently, $\hat g = (\det A)^{-\frac{1}{2}} gA^{-1}$. Notice that $A$ is hermitean  w.r.t. both $g$ and $\hat g$. 

We will say that  $g$ and  $A$ are c-compatible, if $A$ is hermitean and  $g$ and $\hat g$ are c-projectively equivalent.  The following result was proved in \cite{DomMik} (cf. the compatibility condition \eqref{main2}).

\begin{theorem}
A K\"ahler metric $g$ and a hermitean operator $A$ are c-compatible if and only if
\begin{equation}
\nabla_u A = \mathrm{pr}_{\mathbb C}  \bigl( u \otimes d\, \tr A \bigr),
\end{equation}
where $\mathrm{pr}_{\mathbb C}$ denotes the orthogonal projection to the subspace of hermitean operators.
\end{theorem}

The explicit formula for $\mathrm{pr}_{\mathbb C}$ is as follows:
$
\mathrm{pr}_{\mathbb C} L = \frac{1}{4} ( L + L^* +  JLJ + JL^*J )
$.

As in the (pseudo)-Riemannian case discussed in Section~\ref{sect3},  the curvature tensor of a K\"ahler metric  $g$ can be naturally considered as an operator 
$$
R:  \uu(g,J) \to  \uu(g,J),
$$
where $\uu(g,J)$ is the unitary Lie algebra associated with the metric $g$ and complex structure $J$. It is a remarkable fact that if $g$ and $A$ are c-compatible, then $R$ satisfies the following relation
$$
[R(X),A]=[X,B]\quad\mbox{for all }X\in \uu(g,J),\label{eq:RicciId}
$$
where $B= \nabla \mathrm{grad}(\mathrm{tr}A)$.   In other words, $R$ is a sectional operator but in the sense of another Lie algebra, namely $\uu(g,J)$ instead of $\so(g)$.  After Theorem~\ref{BMK}, this property does not look very surprising. 
Here we discuss in brief  just one relatively small part of our paper  \cite{BMR} in order to explain how this property of R can be used in c-projective geometry.

The paper  \cite{BMR} concerns two problems:  local description of c-projectively equivalent metrics and Yano-Obata conjecture which states that essential c-projective vector fields\footnote{An essential c-projective vector  field is defined as a vector field whose flow preserves $J$-planar curves but changes the connection.} may exists on a compact K\"ahler manifold $M$ only in one very special case, namely, if  $M=\mathbb CP^n$ with the standard Fubini--Study metric. 

The proof of the Yano-Obata conjecture is based on our local description of c-projectively equivalent metrics but  the main issue is ``how to pass'' from {\it local}  explicit  formulas for $g_{ij}(x)$  (which become very special  if $g$ admits an essential  c-projective vector field) to {\it global} conclusions.   The main difficulty is that  $g_{ij}$ itself has no simple scalar invariants, like  e.g. eigenvalues.   However such invariants can be constructed from the curvature tensor.  Indeed, if we think of $R$ as an operator defined  on $\uu (g,J)$, then we can consider its eigenvalues as scalar functions on $M$.  Since $M$ is compact, these functions must be bounded and we may try to check this condition  by using our local formulas.  The next problem, however, is computational: how to find explicitly the eigenvalues of such a complicated tensor as $R$?    That is where the properties of sectional operators come into play.  Proposition \ref{prop6} (more precisely, its unitary analog proved in \cite{BMR})  gives a very simple formula for the eigenvalues.  Analysing  these eigenvalues (explicitly found by means of  Proposition \ref{prop6})  has been an important part of our proof.

As a conclusion,  just a few words about  further possible applications of  sectional operators.  As was pointed out in  Section~\ref{sect1},   sectional operators $R:\goth h \to \goth h$ can be naturally defined for any $\Z_2$-graded  Lie algebra $\goth g= \goth h + \goth v$.   The above discussion  shows that in the case $\goth g = \gl(n,\mathbb C)$ and $\goth h= \uu (p,q)$,  the corresponding sectional operators admit a very natural geometric interpretation.   What happens for other $\Z_2$-grading?  Do these operators relate to any interesting geometric structures?    

In his recent paper \cite{Boubel},  C.\,Boubel  has obtained a classification of covaraintly constant $(1,1)$-tensor fields not only on pseudo-Riemannian, but also on K\"ahler and hyper-K\"ahler manifolds of arbitrary signature.   Can we generalise formulas \eqref{rxmin1}, \eqref{rxmin2} and \eqref{r3}
to construct, in a similar way, examples of  K\"ahler and hyper-K\"ahler manifolds with holonomy algebras 
$\goth z_A \cap \uu (p,q)$ and $\goth z_A \cap \mathrm{sp}\, (p,q)$?

\end{document}